\documentclass[11pt]{article}
\usepackage{amsmath,amsthm,amssymb,hyperref, color,fullpage}
\usepackage{blkarray}

\newcommand{\remove}[1]{}

\makeatletter
\newtheorem*{rep@theorem}{\rep@title}
\newcommand{\newreptheorem}[2]{%
\newenvironment{rep#1}[1]{%
 \def\rep@title{#2 \ref{##1}}%
 \begin{rep@theorem}}%
 {\end{rep@theorem}}}
\makeatother

\hypersetup{
	colorlinks=true,
	linkcolor=blue,%
	citecolor=blue,
	linktoc=page
}
\newcommand\numberthis{\addtocounter{equation}{1}\tag{\theequation}}

\newtheorem{thm}{Theorem}[section]
\newreptheorem{thm}{Theorem}
\newtheorem{claim}[thm]{Claim}
\newtheorem{lem}[thm]{Lemma}
\newtheorem{define}[thm]{Definition}

\newtheorem{conjecture}[thm]{Conjecture}

\newtheorem{fact}[thm]{Fact}

\def\F{{\mathbb{F}}}

\def\Z{{\mathbb{Z}}}

\def\R{{\mathbb{R}}}

\def\C{{\mathbb{C}}}

\def\P{{\mathbb{P}}}

\def\Ind{{\mathbf{1}}}

\def\_{\,\,\,\,\,}

\def\rank{\textsf{rank}}
\def\crank{\textsf{crank}}

\newcommand{\eps}{\epsilon}
\usepackage{leftidx}

\usepackage{tikz-cd}

\begin{document}

\title{Proof of the Kakeya set conjecture over rings of integers modulo square-free $N$}
 \author{Manik Dhar\thanks{Department of Computer Science, Princeton University. Email: \texttt{manikd@princeton.edu}. } \and  Zeev Dvir\thanks{Department of Computer Science and Department of Mathematics,
 Princeton University. 
 Email: \texttt{zdvir@princeton.edu}. } }

\date{}
\maketitle

\begin{abstract}
    A Kakeya set  $S \subset (\Z/N\Z)^n$ is a set containing a line in each direction. We show that, when $N$ is any square-free integer, the size of the smallest Kakeya set in $(\Z/N\Z)^n$ is at least $C_{n,\eps}  N^{n - \eps}$ for any $\eps$ -- resolving a special case of a conjecture of Hickman and Wright. Previously, such bounds were only known for the case of prime $N$. We also show that the case of general $N$ can be reduced to lower bounding the $\F_p$ rank of the incidence matrix of points and hyperplanes over $(\Z/p^k\Z)^n$.
\end{abstract}




\section{Introduction}
Given a finite abelian ring $R$, we consider the space of $n$-tuples over $R$, denoted $R^n$. In this space we may define a line in direction $b \in R^n \setminus \{0\} $ to be a subset of the form $\{ a + tb | t\in R\}$ where $a\in R^n$.  We denote the set of directions in $R^n$ (projective space) by $\P R^{n-1}$. For now we will postpone the precise definition of $\P R^{n-1}$  to a later stage. A Kakeya set in $R^n$ is a set containing a line in every direction:

\begin{define}[Kakeya set]
A set $S \subset R^n$ is said to be a Kakeya set if given any direction $b\in \P R^{n-1}$ there exists a point $a\in R^n$ such that the line $\{a+tb|t\in R\}$ is contained in $S$.
\end{define}

The question of lower-bounding the size of the smallest Kakeya set over finite fields was initially raised by Wolff \cite{wolf1999} as a possible approach for attacking the notorious Euclidean Kakeya conjecture in $\R^n$, and later found other applications including in theoretical computer science (see \cite{dvir2010incidence} for a survey of those). Wolff's conjecture ,later known as the finite field Kakeya conjecture stated that, the size of the smallest Kakeya set in $\F_q^n$ should be at least $C_n q^n$ for some constant $C_n$ depending only on the dimension. Coming from the Euclidean problem, one typically  thinks of $n$ as fixed and $q$ growing, however, later applications deal with other scenarios and require  a more accurate control of the constant. Wolff's conjecture was proved using the polynomial method in \cite{dvir2009size} with subsequent improvements given in  \cite{saraf2008,dvir2013extensions} culminating in the following bound.

\begin{thm}[\cite{dvir2013extensions}]\label{thm:TightFqbound}
Let $\F_q$ denote a finite field of  order $q$ and let  $S \subset\F_q^n$ be a Kakeya set. Then
$$|S| \geq \frac{q^n}{\left(2-1/q\right)^n} $$
\end{thm}

When $R$ is not a field much less is known. The problem of lower bounding the size of Kakeya sets for the rings $\Z/p^k\Z$ and $\F_q[x]/\langle x^k\rangle$ was first proposed in \cite{ellenberg2010kakeya} as a step in the direction of the Euclidean problem as these rings contain `scales' in a way that does not exist over a finite field and is reminiscent of the real numbers. While the additive combinatorics techniques that preceded the polynomial method \cite{Bou99,katz2002new}  work over any abelian ring, they currently only lead to bounds of the form $|S| \geq |R|^{\alpha n}$ with the  $\alpha<0.6$. Another, more recent, work to study Kakeya sets (and related operators) over finite rings is  \cite{hickman2018fourier} in which a connection between bounds for Kakeya sets over the rings  $\Z/p^k\Z$  and  the Minkowski dimension of $p$-adic Kakeya sets is established. For the two dimensional case, when $R = \F_q[x]/\langle x^k \rangle$ or $ R = \Z/p^k\Z$, Dummit and Hablicsek showed a (tight) bound of $|S| \geq |R|^2/2k$ in \cite{dummithablicsek2013}. 

Similar to the Euclidean setting, Kakeya sets with Haar measure $0$ can be constructed for the ring of $p$-adic integers and the power series ring $\F_q[[x]]$. The constructions can be found in \cite{dummithablicsek2013,Fraser_2016,CML_2018__10_1_3_0,hickman2018fourier}. As in the Euclidean setting, we want to bound the Minkowski dimension of Kakeya sets for these rings which is connected to the size of Kakeya sets in $\Z/p^k\Z$ and $\F_q[x]/\langle x^k \rangle$.

The Kakeya conjecture  over the rings $R = \Z/N\Z$ was stated  in \cite{hickman2018fourier} as follows. As seen above, for finite fields  the loss of $\eps$ in the exponent is not necessary. However, as we shall see later, for composite $N$ we must allow it.

\begin{conjecture}[Kakeya set conjecture over $\Z/N\Z$]\label{conj:kakeyaR}
For all $\eps > 0$ and integers $n$ there exists a constant $C_{n,\eps}$ such that any Kakeya set $S \subset (\Z/N\Z)^n$ satisfies 
$$ |S| \geq C_{n,\eps} \cdot N^{n-\eps}.$$
\end{conjecture}

 Already when $N=p_1\cdot p_2$ is a product of two primes of roughly the same magnitude, the polynomial method fails to work. One way to see this is to notice that any  polynomial  over $R = \Z/p_1 p_2 \Z$ has degree at most $\max\{ p_1 - 1, p_2 - 1\} \approx N^{1/2}$ in each variable. This limits the dimension of the space of polynomials to $\approx N^{n/2}$ and prevents us from interpolating a non-zero polynomial vanishing on $S$, when $S$ is larger than that dimension (which is the first step in the polynomial method).

Our main contribution is a proof of Conjecutre \ref{conj:kakeyaR} for square-free integers $N$. At this point we should say what is our definition of projective space for these rings as this will determine the definition of Kakeya sets. For $N = p_1\cdots p_r$ a product of $r$ distinct primes, the Chinese remainder theorem gives us $$(\Z/N\Z)^n \cong  \F_{p_1}^n \times \cdots \times \F_{p_r}^n$$ under the natural isomorphism $x \mapsto (x \mod p_i)_{i=1}^r$. For each $i$ the projective space $\P \F_{p_i}^{n-1}$ is defined to be the set of all non-zero vectors in $\F_{p_i}^{n}$ up to scaling. Finally, we take the projective space to be
$$ \P (\Z/N\Z)^{n-1} =  \P \F_{p_1}^{n-1} \times \cdots \times \P \F_{p_r}^{n-1} .$$  In other words, a direction in $(\Z/N \Z)^n$ is represented by a vector $b \in (\Z/N \Z)^n$ that is non-zero modulo $p_i$ for all $i$ and we identify two directions if they can be obtained from one another by scaling with an invertible  ring element.\footnote{ In both  \cite{ellenberg2010kakeya,hickman2018fourier}, where the emphasis was on rings such as   $\Z/p^k \Z$ or $\F_q[x]/x^k$, the definition of projective space over a ring $R$ require the direction $b$ to have at least one invertible coordinate. This definition leads to the same notion of projective space as ours for the rings $\Z/p^k \Z$ or $\F_q[x]/x^k$. However, in the case of $\Z/N\Z$ where $N$ has more than one distinct prime factors their definition is different than ours. In our definition a direction might have all coordinates as zero-divisors. For example, in the case of composite square free $N$ if the reductions $b \mod p_i$ have disjoint supports we get lines with directions represented by non-zero divisors in each co-ordinate. Requiring at least one invertible coordinate leads, in the case of composite square free $N$, to a definition that is basis dependent and less natural. } 

\begin{thm}[Kakeya bound for square-free $N$]\label{thm:SquarefreeNtheorem}
Let $N=p_1\hdots p_r$ be a product of $r$ distinct primes and set $R = \Z/N\Z$. Any Kakeya Set $S\subseteq R^n$ satisfies
$$|S| \geq \frac{N^n}{\prod\limits_{i=1}^r (2-1/p_i)^{n}} \geq \frac{N^n}{2^{rn}}.$$
\end{thm}

Since the number of factors of $N$ satisfies $r = O( \log N / \log \log N)$ (indeed the asymptotics of $r$ are known~\cite{zbMATH02610444})  we see that the expression in the theorem is lower bounded by $N^{n - O(n/\log \log N)}$ and so it indeed proves Conjecture \ref{conj:kakeyaR}.

The tightness of  the bound in Theorem \ref{thm:SquarefreeNtheorem} can be demonstrated by taking the product (via the Chinese remainder theorem) of the best known constructions in $\F_{p_i}^n$ for each factor $p_i$. In \cite{saraf2008} it was shown  that, for any prime $p$, there are Kakeya sets in $\F_p^n$  of size bounded above by $p^n/2^{n-1} + Cp^{n-1}$, where $C$ is an absolute constant. Taking the product  one obtains the following.

\begin{thm}\label{thm:squarefreeUpper}
Let $N=p_1\hdots p_r$ be a product of distinct primes and $n$ an integer. There exist  a Kakeya Set $S\subseteq (\Z/N\Z)^n$ with 
$$|S| \leq \prod\limits_{i=1}^r \left(\frac{p_i^n}{2^{n-1}}+Cp_i^{n-1}\right),$$
where $C>1$ is an absolute constant.
\end{thm}

Hence, when all the prime factors $p_i$ of $N$ are sufficiently large (also as a function of $r$), we see that the main term in the construction is off by at most a factor of $2^r = N^{o(1)}$ from the lower bound of Theorem \ref{thm:SquarefreeNtheorem}. Notice that, while the upper bound is obtained from a product of Kakeya sets modulo each prime factor, a general Kakeya set in $(\Z/N\Z)^n$ might not have this product structure (otherwise the proof of the lower bound would be trivial).

Our proof of Theorem \ref{thm:SquarefreeNtheorem} also outlines a specific problem whose solution could lead to a bound for general (non square-free) modulo $N$. Consider the $p^{kn} \times p^{kn}$ matrix $W_{p^k,n}$ whose rows/columns are indexed by elements of $R^n$ where $R = \Z/p^k\Z$ and whose $(x,y)$'th entry is $1$ if $\langle x,y\rangle=0 \mod p^k$ and $0$ otherwise. We call $W_{p^k,n}$ the {\em point-hyperplane incidence matrix} of $R^n$. Since $W_{p^k,n}$ has zero-one entries, we can view it as a matrix over any field and, in particular, compute its rank over $\F_p$. We show that this rank lower bounds the size of any Kakeya set.
\begin{thm}\label{thm:reduceRankpk}
Given a prime $p$ and integers $k,n$, every Kakeya set $S$ in $(\Z/p^k\Z)^n$ satisfies
\em{$$|S|\ge \text{rank}_{\F_p}(W_{p^k,n}).$$}
\end{thm}

Hence, proving a lower bound of, say, $p^{(1-\eps)kn}$ for small $\eps>0$ on the rank of the incidence matrix $W_{p^k,n}$ would lead to new bounds for Kakeya sets in $(\Z/p^k\Z)^n$.  Furthermore, using our techniques, these bounds will then imply the appropriate bounds for $N$ which is a product of prime powers. The work of  \cite{goethalsDel1968,macwilliams1968p,smith1969p} shows such rank bounds hold for $W_{p,n}$ when $k=1$ or when $R$ a finite field. We will use these bounds to prove a slightly weaker bound leading up to Theorem \ref{thm:SquarefreeNtheorem}. Currently we are able to show that the rank of $W_{p^k,n}$ is only larger than $\approx p^{kn/2}$ which does not lead to any non-trivial bounds on the size of Kakeya sets. The matrix $W_{p^k,n}$, which is  referred to in the literature  as the incidence matrix of {\em Hjelmslev spaces} was shown to have full rank over the rational numbers \cite{LandjevV14} but the rank over $\F_p$ seems to be open. We note that our reduction is only in one direction -- showing that $W_{p^k,n}$ has low rank would not imply the existence of small Kakeya sets using our theorem.

\subsection{Overview of the proof}
Our proof consists of two main parts. The first gives a new formulation of the polynomial proof for Kakeya sets over finite fields (in our case, prime $N$). Our proof relies on the same underlying principles of the polynomial method but uses them in a way that gives us  more control. The second part uses this modified proof for general square-free $N$ by inducting on the number of prime factors. We now describe each part in more detail. 

Consider a Kakeya set  $S \subset \F_p^n$. The first new idea in the proof is to replace the size of the set $S$ with the rank of a $0-1$ matrix $M_S$ we call the {\em line-matrix} of $S$. This matrix has a row for each direction $b \in \P\F_p^{n-1}$ and that row is the indicator vector $\Ind_{L(b)}$ for a line $L(b) \subset S$ in direction $b$. That is,  each column of $M_S$ is indexed by some $x\in \F_p^n$ and the $b$'th row has ones in positions indexed by the points in $L(b)$ and zeros everywhere else. It is not hard to show (and proven in Lemma \ref{lem:rankSizeLine}) that, over any field, $$\rank(M_S) \leq |S| \leq p \cdot \rank(M_S)$$ and so the rank is a good proxy for $|S|$ (for the upper bound we require that $S$ is, in some sense, a minimal Kakeya set).  We will bound the rank of $M_S$ over $\F_p$ by constructing two fixed matrices (independent of $S$) which we denote for now by $A$ and $B$ such that $A$ has high rank and $A = M_S \cdot B$. Since the rank of $M_S \cdot B$ is at most the rank of $M_S$ we get that $|S| \geq \rank(A)$. We leave the description of the matrices $A$ and $B$ (which involves polynomials) to the technical sections as these are not needed to explain the second part of the proof. One comment is that the above outline is only enough to prove a slightly weaker version of Theorem \ref{thm:SquarefreeNtheorem} (with the ``right'' exponent $n$ but worse dependence on $r$). To prove the tighter bound as in the theorem we need to work with a variant of $M_S$ in which each line $L(b)$ has many rows associated with it, each supported on $L(b)$ but with different non-zero values. The construction of the fixed matrices $A$ and $B$ is also different and uses the extended variant of the polynomial method using high order derivatives (as in \cite{dvir2013extensions}). We present both the simplified and full proof in the technical sections below.

With the first part in place, we can now describe the case of composite $N$. For simplicity, assume $N = p \cdot q$ is a product of two primes and let $R = \Z/N\Z$. Notice that, by the Chinese remainder theorem, $R^n \cong \F_{p}^n \times \F_{q}^n$ and any line $L(b)$ in direction $b \in \P R^{n-1}$ is a Cartesian product of a line $L_p(b) \subset \F_{p}^n$ and a line $L_q(b) \subset \F_q^n$. Notice  that each of the lines $L_p(b), L_q(b)$ might depend on both the $b \text{ }(\text{mod } p)$ part and the $b \text{ }(\text{mod } q)$ part of the direction $b$ (otherwise our lives would be much easier as $S$ would be a product of a Kakeya set in $\F_{p}^n$ and a Kakeya set in $\F_{q}^n$). We construct the line matrix $M_S$ as before, working  over the field $\F_{p}$ and treating the rows of $M_S$ as elements in the tensor product $\F_{p}^{p^n} \otimes \F_{p}^{q^n}$. That is, each row is a function from $R^n $ to $\F_{p}$. From the above discussion,  each row $\Ind_{L(b)}$ is the tensor product of $\Ind_{L_p(b)}$ and $\Ind_{L_q(b)}$. We now recall the matrices  $A$ and $B$ from the first part so that, $A = M_{T} \cdot B$ for any Kakeya set $T \subset \F_{p}^n$.  The final step of the proof is multiplying $M_S$ by the Kronecker product $B \otimes I_{q^n}$, where $I_{q^n}$ is a $q^n \times q^n$ identity matrix and analysing the dimension of the space spanned by the rows. This requires both the rank bound on $A$ as well as the inductive bound on Kakeya sets over $\F_q^n$ (which imply rank bounds on the corresponding line-matrix over any field, including $\F_{p}$).

To prove Theorem~\ref{thm:reduceRankpk}  we would ideally like to construct a matrix $B$ such that $W_{p^k,n} = M_S \cdot B$ (which would prove the theorem by the above discussion). While we are not able to directly do that, we are able to construct a {\em complex} matrix $B$ such that $ M_S \cdot B$ has the same support as $W_{p^k,n}$ and whose non-zero entries  are all complex roots of unity of order $p^k$. We then show that the complex rank of such a matrix is lower bounded by the $\F_p$ rank of $W_{p^k,n}$.

\subsection{Organization}
In Section~\ref{sec:finitefield} we (re)prove the finite field Kakeya conjecture using the rank of the line matrix $M_S$. In Section \ref{sec:warmupSqfree} we show how to handle multiple prime factors by proving a weaker version of Theorem~\ref{thm:SquarefreeNtheorem} for the special case $N = pq$.
In Section~\ref{sec:sqfreegen} we prove Theorem \ref{thm:SquarefreeNtheorem} in full generality by adding the use of high order derivatives. In Section~\ref{sec:powConj}  we discuss the case of prime powers $R = \Z/p^k\Z$  and prove Theorem \ref{thm:reduceRankpk}.

\subsection{Acknowledgements}
The authors would like to thank Peter Sin and Ivan Landjev for helpful comments. Research supported by NSF grant DMS-1953807.

\section{Warm-up 1: Reproving the finite field bound }\label{sec:finitefield}

We start be defining the line-matrix $M_S$ associated with a Kakeya set $S \subset R^n$.

\begin{define}[Line matrix of $S$]\label{def:linematrix}
The line matrix $M_S$ for a Kakeya set $S$ in $R^n$ is a matrix with $0,1$ entries where the columns are indexed by points in $R^n$ and the rows are indexed by directions $b\in \P R^{n-1}$ and the row corresponding to $b$ is the indicator vector $\Ind_{L(b)}\in \{0,1\}^{|R^n|}$ of a line $L(b)$ in direction $b$ contained in $S$ (if there is more than one such line, we pick the first one in some pre-determined order). 
\end{define}

We note as $M_S$ is a matrix with $0,1$ entries we can treat it as a matrix over any field $\F$. Given any integer matrix $M$ we let $\textrm{rank}_\F(M)$ refer to the rank of the matrix $M$ over the field $\F$. 

\begin{lem}[Rank-Size relation]\label{lem:rankSizeLine}
Let $S\subseteq R^n$ be a Kakeya set and $\F$ a field. Then
$$ |S| \ge \rank_{\F}(M_S).$$
Furthermore, if $S'$ is the set of non-zero columns in $M_S$ (by identifying the columns by their indices in $R^n$ we see that $S'$ is itself also a Kakeya set in $R^n$) then
$$\rank_\F(M_S)\ge \frac{|S'|}{|R|}.$$
\end{lem}
\begin{proof}
The lower bound on $|S|$ is trivial since all rows are supported on elements of $S$. To prove the other direction we iteratively pick lines from $S'$ as follows. We first start with a line $L_1$. After picking lines $L_1,L_2,\hdots, L_t$ we pick a line $L_{t+1}$ which is not completely contained in the union $\bigcup_{i=1}^t L_i$. The size of the union $\bigcup_{i=1}^t L_i$ is at most $|R|t$ so, since $S'$ is defined to be the union of the lines forming the rows of $M_S$, we can continue this procedure as long as $t|R|<|S'|$. This will gives us a set of lines $L_1,\hdots, L_r$ where $r=\lceil |S'|/|R| \rceil$ with the property that the line $L_{t+1}$ is not completely contained in $\bigcup_{i=1}^t L_i$ for all $t<|S'|/|R|$. The vectors $\Ind_{L_i}$ are clearly linearly independent since they correspond to an upper triangular matrix after changing the basis using an appropriate permutation matrix.
\end{proof}

From now on we will focus on giving a lower bound on the rank of $M_S$ for a Kakeya set $S \subset \F_p^n$ with $p$ prime. As outlined in the proof overview, we are looking to construct two matrices $A$ and $B$ so that $A = M_S \cdot B$ for any Kakeya set $S$ and such that $A$ has high rank. Both $A$ and $B$ will be related to the point-hyperplane incidence matrix which we now define. By a hyperplane we mean a subset $H_b \subset \F_p^n$ of the form $H_b = \{ a \in \F_p^n, \,|\, \langle b,a\rangle = 0 \}$.  We denote by $\overline{H}_b = \{ a \in \F_p^n, \,|\, \langle b,a\rangle \neq 0 \}$ the complement of the hyperplane.

\begin{define}[Point-hyperplane incidence matrix]\label{def:pointhyperplane}
Given a prime $p$ and a natural number $n$ we define the point-hyperplane incidence matrix $W_{p,n}$ to be the $p^n \times p^n$ matrix whose columns are the indicator vectors of the hyperplanes $\Ind_{H_b}$ over all $b \in \F_p^n$ and the rows are indexed by points in $\F_p^n$. Notice that each row/column of $W_{p,n}$ (except for the one indexed by zero) is repeated $p-1$ times as scaling by a non-zero field element does not affect whether or not the inner product is $0$.
\end{define}

Our proof will rely on the following  simple but useful property of this matrix.

\begin{lem}[Action of $W_{p,n}$ on lines]\label{lem:centeredHypAct}
 Let $L\subseteq \F_p^n$ be a line in direction $b\in \P\F_p^{n-1}$ and let $\Ind_L \in \F_p^{p^n}$ be its (row) indicator vector. Then, over the field $\F_p$, we have
 $$\Ind_L\cdot W_{p,n} = \Ind_{\overline{H}_b}.$$
 Hence, the product only depends on the {\em direction} of the line $L$.
\end{lem}
\begin{proof}
The coordinate of $\Ind_L\cdot W_{p,n}$ indexed by $a\in \F_p^n$ is the inner product of $\Ind_L$ and $\Ind_{H_a}$ and so is equal to the size (mod $p$) of the intersection $L \cap H_a$. For  $a=0$  the intersection size is $|L| = p$ which is zero mod $p$. This is also the $0$ entry in $\Ind_{\overline{H}_b}$ since $0 \not\in \overline{H}_b$. Now suppose $a \neq 0$. The size of the intersection of a line $L$ with a non-trivial hyperplane $H_a$ can have one of three values. If the direction $b$ of the line $L$ is not in $H_a$ then  $|L \cap H_a|=1$. If $b \in H_a$ then $|L \cap H_a|$ can be either 0 or $|L|$ which are both equal to 0 modulo $p$. Hence, over $\F_p$ we have $(\Ind_L \cdot W_{p,n})_a = (\Ind_{\overline{H}_b})_a$.
\end{proof}

We will also need a bound on the rank of the matrix $W_{p,n}$. These matrices have been studied in the coding community since the 1960's and their rank has been computed via several methods. In particular, the following is a special case of results appearing in ~\cite{goethalsDel1968,macwilliams1968p,smith1969p}. The proofs involve using studying the point-hyperplane incidence matrix for the projective space $\P \F_p^n$, identifying $\P \F_p^n$ with the cyclic group $\F_{p^{n+1}}^\times/\F_p^\times$, and using its representation theory over $\F_p$.

\begin{thm}[$\F_p$-rank of $W_{p,n}$]\label{thm:pRankSubs}
Let $W_{p,n}$ be the point-hyperplane incidence matrix of $\F_p^n$. Then
$$ \rank_{\F_p}(W_{p,n}) = \binom{p+n-2}{n-1}+1.$$
\end{thm}

We will conclude this section by demonstrating how the information obtained so far for $W_{p,n}$ can be used to give a bound on Kakeya sets over prime order finite fields.\footnote{One can generalize the proof using $W_{p,n}$ also to fields of  size $p^t$ but the resulting  bounds are not as good as the ones obtained directly from the polynomial methods.} 

\begin{thm}\label{thm:oldnewproof}
Let $S \subset \F_p^n$ be a Kakeya set. Then
$$|S|\geq \binom{p+n-2}{n-1}.$$
\end{thm}
\begin{proof}
Working over the field $\F_p$, let $M_S$ be the line matrix of $S$ and let $$ A = M_S \cdot W_{p,n}.$$
By Lemma~\ref{lem:centeredHypAct} we have that the row of $A$ indexed by $b \in \P \F_p^{n-1}$ is the indicator vector $\Ind_{\overline{H}_b}$.  Let $A' = J - A$  where $J$ is a matrix with all entries equal to one. Then, the matrix $A'$ has the same rows (without repetition) as those of $W_{p,n}$ and hence the same rank. Since $J$ is rank one, we get that $$ \rank_{\F_p}(A) \geq \rank_{\F_p}(W_{p,n}) - 1 = \binom{p+n-2}{n-1}. $$ Since $$|S| \geq \rank_{\F_p}(M_S) \geq \rank_{\F_p}(M_S \cdot W_{p,n})$$ we get the claimed bound.
\end{proof}

The relationship between the above proof and the polynomial method proof appearing in \cite{dvir2009size} is somewhat elusive at this point and will become clearer when we prove the stronger bound appearing in  Theorem~\ref{thm:SquarefreeNtheorem}. However, the proof of this section could be taken on its own as a `new' proof of the finite field bound which does not use polynomials in any explicit way (but, in turn, relies on the rank bound of Theorem~\ref{thm:pRankSubs}). The bound obtained above only gives an exponent of $n-1$ instead of $n$. However, this can be amplified to $n-\eps$ for all $\eps$ using a standard tensoring trick (see Lemma~\ref{lem:productkakeya}).
\section{Warm-up 2: Product of two primes}\label{sec:warmupSqfree}
In this section we show how the proof of the finite field case given in the previous section allows us to work with composite (square-free) modulus. For the sake of simplicity we deal with the case of two distinct primes as it already contains all the technical details of the general case. 

First, we define the Kronecker Product of two matrices and its relation to the tensor product. We denote $[r] = \{1,2,\ldots,r\}$.

\begin{define}[Kronecker Product of two matrices]
Given a field $\F$ and two matrices $M_A$ and $M_B$ of sizes $n_1\times m_1$ and $n_2\times m_2$ corresponding to linear maps $A:\F^{n_1}\rightarrow \F^{m_1}$ and $B:\F^{n_2}\rightarrow \F^{m_2}$ respectively, we define the Kronecker product $M_A\otimes M_B$ as a matrix of size $n_1n_2\times m_1m_2$ with its rows indexed by elements in $[n_1]\times [n_2]$ and its columns indexed by elements in $[m_1]\times [m_2]$ such that
$$M_A\otimes M_B ((r_1,r_2),(c_1,c_2))=M_A(r_1,c_1)M_B(r_2,c_2),$$
where $r_1\in [n_1],r_2\in [n_1],c_1\in [m_1]$ and $c_2\in [m_2]$. $M_A\otimes M_B$ corresponds to the matrix of the linear map $A\otimes B: \F^{n_1}\otimes \F^{n_2}\cong \F^{n_1n_2}\rightarrow \F^{m_1}\otimes \F^{m_2}\cong \F^{m_1m_2}$.
\end{define}

We will need the following simple property of Kronecker products which follows from the corresponding property of the tensor product of linear maps.

\begin{fact}[Multiplication of Kronecker products]\label{fact:multOfKronecker}
Given matrices $A_1,A_2,B_1$ and $B_2$ of sizes $a_1\times n_1$, $a_2\times n_2$, $n_1\times b_1$ and $n_2\times b_2$ we have the following identity,
$$(A_1 \otimes A_2) \cdot (B_1\otimes B_2)=(A_1\cdot B_1)\otimes (A_2\cdot B_2).$$
\end{fact}

For the rest of this section, let $N = pq$ be a product of distinct primes and denote $R = \Z/N\Z$.  Recall that, via the Chinese remainder theorem, we have a natural isomorphism between $R$ and $\F_p \times \F_q$ which extends to $R^n \cong \F_p^n \times \F_q^n$. We will work in the tensor product $\F_p^{p^n} \otimes \F_p^{q^n}$ which we will identify with the space $\F_p^{N^n}$. If we consider $v \in \F_p^{p^n}$ as a function $v : \F_p^n \mapsto \F_p$ and $u \in \F_p^{q^n}$ as a function $v : \F_q^n \mapsto \F_p$ then their tensor product $v \otimes u \in \F_p^{N^n}$ is the function $v \otimes u : R^n \mapsto \F_p$ defined by $(v \otimes u)(x_p,x_q) = v(x_p) \cdot u(x_q)$ where $(x_p,x_q) \in \F_p^n \times \F_q^n$ is a general element in $R^n$ via the Chinese remainder theorem.

We will need the following simple lemma on the rank of certain sets of vectors inside the tensored space. 

\begin{lem}\label{lem:tensorProductSpaceBound}
Let $V$ and $U$ be finite dimensional vector spaces over an arbitrary  field $\F$. Let $A=\{v_1,v_2,\hdots,v_n\} \subseteq V$ be a set of linearly independent vectors and $B_1,B_2,\hdots,B_n \subseteq U$ be subsets such that each  $B_i$ spans a subspace of dimension at least $k$. Then the space spanned by the vectors $\bigcup_{i=1}^n\{v_i\otimes y| y\in B_i\}$ has dimension at least $nk$ in $V\otimes U$
\end{lem}
\begin{proof}
Let $C=\{c_1,\hdots,c_m\}$ be some basis of $U$ where we let $m$ be the dimension of $U$. Then $v_i\otimes c_j$ where $i\in [n]$ and $j\in [m]$ form a linearly independent set of vectors in $V\otimes U$. For each set $B_i$ we can find a set of $k$ linearly independent vectors $B'_i=\{b^i_1,\hdots,b^i_k\}$ all of which can be written as a linear combination of elements in $C$.

We will show the set of vectors $\bigcup_{i=1}^n\{v_i\otimes y| y\in B'_i\}$ is linearly independent to prove the lemma. Let us consider a linear combination of these vectors which equals $0$,
$$\sum_{i=1}^n \sum_{j=1}^k \alpha_{i,j} v_i\otimes b^i_j=0,$$
where $\alpha_{i,j}$ are scalars. This means,
\begin{align}
\sum\limits_{i=1}^n v_i\otimes \left(\sum_{j=1}^k \alpha_{i,j}  b^i_j\right)=0. \label{eq:tensorProductLemma}
\end{align}

Consider the linear projection operator $P_i$ defined on $\textrm{span}\{v_i, i \in [n]\} \otimes U$ as,
$$P_i(v_j\otimes c_k) =
\left\{
	\begin{array}{ll}
		c_k  & \mbox{if } j = i \\
		0 & \mbox{if } j \ne i
	\end{array}
\right. .$$
Applying $P_i$ on \eqref{eq:tensorProductLemma} gives us,
$$ \sum_{j=1}^k \alpha_{i,j} b^i_j=0.$$
As $b^i_1,\hdots,b^i_k$ are linearly independent we have $\alpha_{i,j}=0$ for all $i$ and $j$.
\end{proof}

We will also need the following simple claim which allows one to amplify a bound of the form $N^{n - c}$ for some constant $c$ to a bound of the form $N^{n-\eps}$ for any $\eps > 0$. 

\begin{lem}\label{lem:productkakeya}
If $S$ is a Kakeya set in $R^n$ where $R=\Z/N\Z$ for square-free $N$, then $S^t\subseteq R^{tn}$ which is the product of $S$ with itself $t$ times is also a Kakeya set in $R^{tn}$.
\end{lem}
\begin{proof}
It is enough to consider the case $t=2$. Let $b \in \P R^{2n-1}$ be some direction and, by abuse of notation, let us think of $b \in R^{2n}$ as some representative of this direction. Write $b = (b',b'')$ where $b'$ is the first $n$ coordinates of $b$ and $b''$ are the last $n$ coordinates (each corresponding to a different copy of $S$). If $N = p_1 \cdots p_r$, for each $i$ we let $b'_i = b' \mod p_i$ and similarly for $b''$. If all $b'_i$ and $b''_i$ are non-zero then we are in a situation where $b'$ and $b''$ are `legal' directions in $\P R^{n-1}$ and so there will be lines $L',L''$ in $S$ in these directions. Therefore the product $L' \times L'' \subseteq S^2$ will contain a line in direction $b$. 

A slightly more delicate case occurs when some of the $b'_i$ or $b''_i$ are zero. In this case, let $L' \subset S$ be a line in some direction $c'$ that agrees with $b'$ modulo all $p_i$ for which $b'_i$ is non zero and take $L'', c''$ in a similar manner. We now have to check that the product $L' \times L''$ contains a line in direction $b = (b',b'')$. Suppose $L' = \{a' + tc' \,|\, t \in R\}$ and similarly $L'' = \{a'' + tc'' \,|\, t \in R\}$. Consider the line in $R^{2n}$ in direction $b = (b',b'')$ through $a = (a',a'')$. A general point on this line looks like $x(t) = (a' + tb',a'' + tb'')$. Now, the set $L' \times L''$ contains all points of the form $y(t',t'') = (a' + t'c',a'' + t'' c'')$. We now have to check that the point $x(t)$ as above is in this product. Let $t' \in R$ be the same as $t$ but with $t' =0  \mod p_i$ for all $p_i$ such that $b'_i=0$ and similarly let $t''$ be the same as $t$ but with $t''=0 \mod p_i$ for all $p_i$ such that $b''_i=0$. We have $t'c' = tb'$ and similarly $t''c'' = tb''$. Therefore, $x(t) = y(t',t'')$ and we are done.
\end{proof}

We are now ready to prove the main result of this section.

\begin{thm}[Kakeya bound in $(\Z/pq\Z)^n$]
Let $p$ and $q$ be distinct primes and let $S \subset (\Z/pq\Z)^n$ be a Kakeya set. Then, for any $\eps >0$ there exists a constant $C_{n,\eps}$ depending only on $n$ and $\eps$ so that 
$$|S|\ge C_{n,\eps} \cdot (pq)^{n-\eps}.$$ 
\end{thm}
\begin{proof}
Let $R=\Z/pq\Z$. All our  vectors and matrices will be over $\F_p$. Given a Kakeya set $S\subseteq R^n$ consider the line matrix $M_S$ associated with $S$ over the field $\F_p$. Our goal is to lower bound 
$\rank_{\F_p}(M_S)$.   For a direction $$b=(b_p,b_q)\in \P\F_{p}^{n-1}\times \P\F_q^{n-1},$$ the row in $M_S$ corresponding to $b$ will be the indicator vector $\Ind_{L(b)}\in \F_p^{|R^n|}$ of a line in direction $b$ contained in $S$ denoted  $$L(b)=L(b_p,b_q) = L_p(b_p,b_q) \times L_q(b_p,b_q)$$ which is itself Cartesian product of lines $L_p(b_p,b_q) \subseteq \F_{p}^n$ in the direction $b_p$ and $L_q(b_p,b_q)\subseteq \F_{q}^n$ in the direction $b_q$. Note, $L_p(b_p,b_q) $ includes $b_q$ because the lines $L_p(b_p,c_1)$ and $L_p(b_p,c_2)$ can be potentially different when $c_1\ne c_2$. Finally, notice that the indicator vector $\Ind_{L(b)}\in \F_p^{|R^n|}$ equals the tensor product $$\Ind_{L(b)} = \Ind_{L_p(b_p,b_q)}\otimes \Ind_{L_q(b_p,b_q)}.$$

Let $W_{p,n}$ be the point-hyperplane  incidence matrix defined in the previous section. Let $I_{q^n}$ be the identity matrix of size $q^n\times q^n$. The rows and columns in $I_{q^n}$ are thought to be indexed by points in $\F_{q}^n$. Consider the Kronecker product $W_{p,n}\otimes I_{q^n}$. We will examine the product $$M_S\cdot (W_{p,n}\otimes I_{q^n}).$$ If we look at the row in $M_S$ indexed by a direction $b=(b_p,b_q)\in \P \F_{p}^{n-1}\times \P\F_q^{n-1}$, the corresponding row in $M_S\cdot (W_{p,n}\otimes I_{q^n})$ is the product,
\begin{align*}
\Ind_{L(b)}\cdot (W_{p,n}\otimes I_{q^n})&=(\Ind_{L_p(b_p,b_q)}\otimes \Ind_{L_q(b_p,b_q)})\cdot (W_{p,n}\otimes I_{q^n})\\
&=(\Ind_{L_p(b_p,b_q)}\cdot W_{p,n})\otimes \Ind_{L_q(b_p,b_q)} && \text{(Using Fact } \ref{fact:multOfKronecker}\text{)}  \\
&=\Ind_{\overline{H}_{b_p}}\otimes \Ind_{L_q(b_p,b_q)}, \numberthis \label{eq:tensorSimpleInd}    
\end{align*}
where we recall $\overline{H}_{c}$ is the set $\{x|\langle x,c\rangle \ne 0\}\subseteq \F_p^n$ where $c\in \P \F_p^{n-1}$ and the last step in \eqref{eq:tensorSimpleInd} uses  Lemma \ref{lem:centeredHypAct}.

Denote the set of  vectors $$V=\{\Ind_{\overline{H}_c}\,|\, c\in \P\F_p^{n-1}\} \subset  \F_p^{p^n}$$ and, for each $c \in \P \F_p^{n-1}$, the set of vectors 
$$B_c=\{\Ind_{L_q(c,b_q)}\,|\,b_q\in \P\F_q^{n-1}\} \subset \F_p^{q^n}.$$ 

From Theorem~\ref{thm:pRankSubs} (following the argument in Theorem~\ref{thm:oldnewproof}) we know that the dimension the space spanned by $V$ is at least $\binom{p+n-2}{n-1}$. Next, fix some $c \in \P\F_p^{n-1}$. The vectors forming $B_c$ are the indicators of a set of lines in every direction in $\F_q^n$. Lemma \ref{lem:rankSizeLine}, combined with the bound on Kakeya sets in $\F_q^n$ (Theorem~\ref{thm:TightFqbound}) imply that $B_c$ will have rank at least $q^{n-1}2^{-n}$ over any field and in particular over $\F_p$. Hence, by Lemma~\ref{lem:tensorProductSpaceBound} we have that the set of vectors $$\left\{ \Ind_{\overline{H}_c}\otimes u\,\,|\,\,c \in \P\F_p^{n-1}, u\in B_c\right\}$$ has rank at least  $$\binom{p+n-2}{n-1} \cdot q^{n-1}2^{-n}.$$ Since these are the rows of $M_S$ after multiplying by a matrix $W_{p,n} \otimes I_{q^n}$ we get that this is also a lower bound on the rank of $M_S$. By Lemma~\ref{lem:rankSizeLine}, we have that $$ |S| \geq \binom{p+n-2}{n-1} \cdot q^{n-1}2^{-n} \geq C_n \cdot N^{n-1}$$ for some constant $C_n$ depending only on $n$. The bound in the theorem now follows from applying the weaker  bound on the $t$-fold Cartesian product $S^t = S \times \ldots \times S \subset R^{nt}$ (which is also a Kakeya set by  Lemma~\ref{lem:productkakeya}) and then use the fact that $|S^t| = |S|^t$.
\end{proof}

\section{The general square-free \(N\) case}\label{sec:sqfreegen}

To get the bound stated in Theorem~\ref{thm:SquarefreeNtheorem} we will generalize the proof structure of Section~\ref{sec:finitefield} to use ideas from the extended polynomial method using high order multiplicities as in \cite{dvir2013extensions}. 

We begin with some definitions and basic results concerning polynomials over finite fields.  We let $\F_p[x_1,\hdots,x_n]_{=d}$ denote the vector space of homogeneous $n$-variate degree $d$ polynomials over $\F_p$ and $\F_p[x_1,\hdots,x_n]_{\le d}$ denote  the space of polynomials of degree at most $d$. We let $$\delta_{n,d}=\binom{n+d-1}{n-1}$$ denote the dimension of the space $ \F_p[x_1,\hdots,x_n]_{=d}$ and $$\Delta_{n,d}=\binom{n+d}{n}$$  the dimension of  $\F_p[x_1,\hdots,x_n]_{\le d}$.  For a tuple $\mathbf{i}\in \Z_{\ge 0}^n$ we define the weight of $\mathbf{i}$ as $$\text{wt}(\mathbf{i})=\sum_{j=1}^n i_j.$$ 

\begin{define}[Hasse Derivatives]
Given a polynomial $f\in \F[x_1,\hdots,x_n]$ for any field $\F$ and an $\mathbf{i}\in \Z_{\ge 0}^n$ the $\mathbf{i}$th {\em Hasse derivative} of $f$ is the polynomial $f^{(\mathbf{i})}$ in the expansion $f(x+z)=\sum_{\mathbf{j}\in \Z_{\ge 0}^n} f^{(\mathbf{j})}(x)z^{\mathbf{j}}$ where $x=(x_1,...,x_n)$, $z=(z_1,...,z_n)$ and $z^{\mathbf{j}}=\prod_{k=1}^n z_k^{j_k}$.  
\end{define}

\begin{define}[Multiplicity]
For a polynomial $f\in \F[x_1,\hdots,x_n]$ and a point $a\in \F^n$ we say $f$ vanishes on $a$ with {\em multiplicity} $m\in \Z$, if $m$ is the largest integer such that all Hasse derivatives of $f$ of weight strictly less than $m$ vanish on $a$. We use $\text{mult}(f,a)$ to refer to the multiplicity of $f$ at $a$.
\end{define}

Notice, $\text{mult}(f,a)=1$ just means $f(a)=0$. Also the number of Hasse derivatives over $\F[x_1,\hdots,x_n]$ with weight strictly less than $m$ is $\Delta_{n,m-1}$. One can also check that for a univariate polynomial $f(x)$ to vanish at $a$ with multiplicity $m$, $f$ must be divisible by $(x-a)^m$. 

We will need an extended Schwartz-Zippel bound~\cite{schwartz1979probabilistic,ZippelPaper} which leverages multiplicities. The proof can be found in \cite{dvir2013extensions}.

\begin{lem}[Schwartz-Zippel with multiplicities]\label{multSchwartz}
Let $f\in \F[x_1,..,x_n]_{\le d}$, with $\F$ an arbitrary field and $d\in \Z$. Then for any finite subset $U\subseteq \F$ ,
$$\sum\limits_{a\in U^n} \text{mult}(f,a) \le d|U|^{n-1}.$$
\end{lem}


We now define a family of linear maps sending a polynomial to a list of its evaluations (with derivatives) over some set. These maps (or more precisely, the matrices representing them) will replace the matrices $W_{p,n}$ used in Section~\ref{sec:finitefield}.

\begin{define}[Evaluation maps]
For a prime $p$, natural numbers $n$ and $m$, given a set $A$ in $\F_p^n$, we let $$\text{EVAL}_A^m:\F_p[x_1,\hdots,x_n]\rightarrow \F^{|A|\Delta_{n,m-1}}$$ refer to the linear map from $\F_p[x_1,\hdots,x_n]$ to the evaluation of all Hasse derivatives of weight strictly less than $m$ over the set $A$.  We treat the points in the co-domain $\F_p^{|A|\Delta_{n,m-1}}$ as column vectors of length $|A|\Delta_{n,m-1}$ indexed by tuples $(x,\mathbf{j})\in A\times \Z_{\ge 0}^n$ with $\text{wt}(\mathbf{j})<m$. The $(x,\mathbf{j})$th entry of $\text{EVAL}_A^m(f)$ for a polynomial $f\in \F_p[x_1,\hdots,x_n]$ is $f^{(\mathbf{j})}(x)$. For singleton sets $\{x\}$ we omit the curly braces and write $\text{EVAL}_x^m$. 
\end{define}

We will now construct matrices $C^k_L$ which will replace the indicators $\Ind_{L(b)}$ used as the rows of $M_S$ in the proofs of the previous sections. Intuitively, the matrix $C^k_L$ for some line $L$ in direction $b$ corresponds to the linear map which takes as input the evaluations (up to some order $m$ depending on $k$) of a polynomial $f$ on the line $L$ and output the evaluation of $f$ (up to order $k$) at the point $b$. This is possible as long as the degree of $f$ is not too big as a consequence of Lemma~\ref{multSchwartz}. 

\begin{lem}[The decoding matrix $C^k_L$]\label{lem:infitDerLine}
Given a prime $p$, numbers $k,n,m\in \Z_{\ge 0}$ where $p|k$ and $m=2k-k/p$, and a line $L\subseteq \F_p^n$ in the direction $b\in \P\F_p^{n-1}$, we can construct a $\Delta_{n,k-1}\times p^n\Delta_{n,m-1}$ matrix $C_L^k$ such that,
\begin{enumerate}
    \item The rows in the matrix $C_L^k$ are indexed by points $\mathbf{i}\in \Z_{\ge 0}^n, \text{wt}(\mathbf{i})<k$ and the columns are indexed by tuples $(x,\mathbf{j})\in \F_p^n\times \Z_{\ge 0}^n$ with $\text{wt}(\mathbf{j})<m$.
    \item The only non-zero columns are the ones corresponding to tuples $(x,\mathbf{j})$ with $x\in L$.
    \item For a polynomial $f\in \F_p[x_1,\hdots,x_n]_{=kp-1}$ we have,
    $$C_L^k\cdot \text{EVAL}_{\F^n_p}^m(f)=\text{EVAL}_{b}^k(f).$$
\end{enumerate}
\end{lem}
\begin{proof}
We will need the following two claims.
\begin{claim}
For homogenous polynomials $f$ of degree $kp-1$ we have 
\begin{align}
    \text{EVAL}^m_L(f)=0\implies \text{EVAL}^k_b(f)=0,\label{eq:LineToDirection}
\end{align}
where $m=2k-k/p$.
\end{claim}
The proof can be found in Theorem 11 (in the arxiv version) and Theorem 3.2 (in the SIAM version) of \cite{dvir2013extensions} and is a consequence of Lemma \ref{multSchwartz}. 
\begin{claim}
Let $A$ be an $n_1 \times w$ matrix and $B$ be an  $n_2 \times w$ matrix both over a field $\F$ and suppose that, for all $x\in \F^w$ we have,
$$Ax=0\implies Bx=0.$$
Then there exists a matrix $C$ of size $n_2\times n_1$ such that $C\cdot A=B$.
\end{claim}
\begin{proof}
For all $x\in \F^w, Ax=0\implies Bx=0$ means the kernel of $A$ is a subset of the kernel of $B$. This means every row of $B$ is spanned by the row space of $A$. This immediately implies that we can construct $C$ such that $CA=B$.
\end{proof}

Combining the two claims lets us construct a matrix $C'$ such that 
\begin{align}
    C'\cdot \text{EVAL}_{L}^m(f)=\text{EVAL}_b^k(f).\label{eq:Cpartialprop}
\end{align} 
The columns in $C'$ are indexed by tuples $(x,\mathbf{j})\in L\times \Z_{\ge 0}$ such that $\text{wt}(\mathbf{j})<m$. We add zero columns to $C'$ corresponding to tuples $(x,\mathbf{j})\in (\F_p^n\setminus L)\times \Z_{\ge 0}^n$ with $\text{wt}(\mathbf{j})<m$. This gives us $C^k_L$. By construction it satisfies the first two properties. The third property follows from \eqref{eq:Cpartialprop}.
\end{proof}

In our proof, it will be convenient to work with the following extension of rank for sets of matrices.

\begin{define}[crank of a set of matrices]
Given a finite set $T=\{A_1,\hdots,A_n\}$ of matrices over a field $\F$ having the same number of columns we let $\crank(T)$ be the rank of the matrix obtained by concatenating all the elements $A_i$ in $T$ along their columns. In other words it is the dimension of the subspace spanned by the vectors in the set $\bigcup_{i=1}^n\{r|r\text{ is a row in }A_i\}$.
\end{define}

We will use a simple lemma which follows from the definition.

\begin{lem}[crank bound for multiplying matrices]\label{lem:crankMatMult}
Given matrices $A_1,\hdots,A_n$ of size $a\times b$ and a matrix $H$ of size $b\times c$ we have
\em{$$\crank\{A_i\}_{i=1}^n\ge\crank\{A_i\cdot H\}_{i=1}^n.$$}
\end{lem}

We now need an extension of Lemma \ref{lem:tensorProductSpaceBound} for this definition.

\begin{lem}[crank bound for tensor products]\label{lem:crankTensorProduct}
Given matrices $A_1,\hdots,A_n$ of size $a_1\times a_2$ over a field $\F$ such that $\crank\{A_i\}_{i=1}^n\ge r_1$ and matrices $B_{i,j}$ over the field $\F$ for $i\in [n]$ and $j\in [m]$ of size $b_1\times b_2$ such that $\crank\{B_{i,j}\}_{j=1}^m\ge r_2$ for all $i\in [n]$ we have,
\em{$$\crank\{A_i\otimes B_{i,j}|i\in [n],j\in[m]\}\ge r_1r_2.$$}
\end{lem}
\begin{proof}
Let $V=\bigcup_{i=1}^n\{w|w\text{ is a row in }A_i\}$ and $U_i=\bigcup_{j=1}^m\{w|w\text{ is a row in }B_{i,j}\}$ for $i\in [n]$. $V$ has rank at least $r_1$ and each of the $U_i$ will have rank at least $r_2$. Using Lemma \ref{lem:tensorProductSpaceBound} we see that the set of vectors $\bigcup_{i=1}^n\{w_1\otimes w_2|w_1\text{ is a row in }A_i,w_2\in U_i\}$ will have rank at least $r_1r_2$. This gives us the desired $\crank$ bound too.
\end{proof}

\subsection{Proof of Theorem~\ref{thm:SquarefreeNtheorem}}
We are now ready to prove our main result restated here for convenience.

\begin{repthm}{thm:SquarefreeNtheorem}
Let $N=p_1\hdots p_r$ for distinct primes $p_1,\hdots,p_r$.  Any Kakeya Set $S\subseteq (\Z/N\Z)^n$ satisfies
$$|S| \geq \frac{N^n}{\prod\limits_{i=1}^r (2-1/p_i)^{n}}.$$
\end{repthm}
\begin{proof}
We will prove this using induction over $r$. When $r=1$ the result is known via Theorem \ref{thm:TightFqbound}.

Let us assume the bound holds for a product of $r$ primes. Let $N=p_1\hdots p_{r+1}$ for $r+1$ distinct primes and $R=\Z/N\Z$. To prove a lower bound let us take a Kakeya set $S$ in $R^n$. For convenience we let $N_0=p_2p_3\hdots p_{r+1}$ and $R_0=\Z/N_0\Z$.

All our matrices and indicator vectors will be over $\F_{p_1}$.
By the Chinese remainder theorem $R^n$ is isomorphic to $\F^n_{p_1}\times R_0^n$. Every direction $b\in \P R^{n-1}$ is represented by a tuple $(b_1,b_0)\in \P\F_{p_1}^{n-1}\times \P R_0^{n-1}$. Any line $L\subseteq R^n$ in direction $b=(b_1,b_0)\in \P\F_{p_1}^{n-1}\times \P R_0^{n-1}$ is a product of lines $L_1\subseteq \F_{p_1}^n$ in direction $b_1$ and $L_0\subseteq R_0^n$ in direction $b_0$. The indicator vector  $\Ind_{L}\in \F_{p_1}^{|R^n|}$ will equal $\Ind_{L_1}\otimes \Ind_{L_0}\in \F_{p_1}^{|\F_{p_1}^n\times R_0^n|}$.



For each direction $b\in \P R^{n-1}$ we must have a line $L(b)$ contained in $S$. If there are many such lines we pick one arbitrarily. The line $L(b)$ will be the product of lines $L_1(b)$ and $L_0(b)$ in $\F_{p_1}^n$ and $R_0^n$ respectively. 

Let us fix a natural number $k$ divisible by $p_1$. For a direction $b$ consider the matrix $C^k_{L_1(b)}$ (given by Lemma \ref{lem:infitDerLine}) over the field $\F_{p_1}$ which will be of size $\Delta_{n,k-1}\times p_1^n\Delta_{n,m-1}$ with $$m=2k-k/p_1.$$ The following claim generalizes Lemma~\ref{lem:rankSizeLine} and allows us to lower bound $|S|$ using a rank bound (in this case the rank of the matrix containing all rows in all $C_{L(b)}^k$).



\begin{claim}[$\crank$-Size relation]\label{lem:rankSizeProject}
\em{$$|S|\binom{m+n-1}{n}\ge \crank\{C^k_{L_1(b)}\otimes \Ind_{L_0(b)}\}_{b\in \P R^{n-1}}.$$}
\end{claim}
\begin{proof}
The columns in $C^k_{L_1(b)}\otimes \Ind_{L_0(b)}$ for all $b$, are indexed by tuples $(x,\mathbf{j})\in R^n\times \Z_{\ge 0}^n$ with $\text{wt}(\mathbf{j})<m$. By Lemma \ref{lem:infitDerLine} we see the non-zero columns in $C^k_{L(b)}\otimes \Ind_{L_0(b)}$ will correspond to tuples with $x\in L_1(b)\subseteq S$. In general in all the matrices in the set $\{C^k_{L_1(b)}\otimes \Ind_{L_0(b)}\}_{b\in \P R^{n-1}}$ the non-zero columns all correspond to points $(x,\mathbf{j})\in S\times \Z_{\ge 0}$ with $\text{wt}(\mathbf{j})<m$. This gives us the required bound.
\end{proof}

Let $E$ be a matrix of size $p_1^n\Delta_{n,m-1}\times \delta_{n,kp_1-1}$ representing the linear map $\text{EVAL}_{\F^n_{p_1}}^m$ restricted to the space $\F_{p_1}[x_1,\hdots,x_n]_{=kp_1-1}$ (with  some arbitrary basis). Given a direction $b_1\in \P\F_{p_1}^{n-1}$, we let $D_{b_1}$ be the $\Delta_{n,k-1}\times \delta_{n,kp_1-1}$ matrix representing the linear map $\text{EVAL}_{b_1}^k$ restricted to the space $\F_{p_1}[x_1,\hdots,x_n]_{=kp_1-1}$. For $b=(b_1,b_0)\in \P\F_{p_1}^{n-1}\times \P R_0^{n-1}$, Lemma \ref{lem:infitDerLine} implies $$C^k_{L(b)}\cdot \text{EVAL}_{\F^n_{p_1}}^m(f)=\text{EVAL}_{b_1}^k(f)$$ for any $f\in \F_{p_1}[x_1,\hdots,x_n]_{=kp_1-1}$. This implies $$C^k_{L(b)}\cdot E=D_{b_1}.$$ Let $I_{N_0^n}$ be the identity matrix of size $N_0^n\times N_0^n$. Using Lemma \ref{lem:crankMatMult} we have,
\begin{align*}
    \crank\{C^k_{L_1(b)}\otimes \Ind_{L_0(b)}\}_{b\in \P R^{n-1}}&\ge \crank\{(C^k_{L_1(b)}\otimes \Ind_{L_0(b)})\cdot (E\otimes I_{N_0^n})\}_{b\in \P R^{n-1}}\\
    &=\crank\{(C^k_{L_1(b)}\cdot E) \otimes \Ind_{L_0(b)}\}_{b\in \P R^{n-1}} && \text{(Using Fact }\ref{fact:multOfKronecker}\text{)}\\
    &=\crank\{D_{b_1}\otimes \Ind_{L_0(b_1,b_0)}\}_{b = (b_1,b_0)\in \P \F_{p_1}^{n-1}\times \P R_0^{n-1}}\numberthis \label{eq:Dbrankeq}.
\end{align*}

To lower bound $\crank\{D_{b_1}\otimes \Ind_{L_0(b_1,b_0)}\}_{(b_1,b_0)\in \P \F_{p_1}^{n-1}\times \P R_0^{n-1}}$ we will use Lemma \ref{lem:crankTensorProduct}. To that end we want to lower bound $\crank\{D_{b_1}\}_{b_1\in \P\F_{p_1}^{n-1}}$ and $\crank\{\Ind_{L_0(c,b_0)}\}_{b_0\in \P R_0^{n-1}}$ for $c\in \P \F_{p_1}^{n-1}$.

\begin{claim}\label{claim:crankCparameterBound}
For all $c\in \P \F_{p_1}^{n-1}$ we have
\em{$$\crank\{\Ind_{L_0(c,b_0)}\}_{b_0\in \P R_0^{n-1}}\ge \frac{N_0^{n-1}}{\prod\limits_{i=2}^{r+1} \left(2-p_i^{-1}\right)^{n}}, $$}

\end{claim}
\begin{proof}

For a fixed $c\in \P\F_{p_1}^{n-1}$ we see that $$\bigcup_{b_0\in \P R_0^{n-1}} L_0(c,b_0)\subseteq R_0^n$$ is a Kakeya set in $R_0^n$ which is a union of lines in every direction. We note $\crank(\{\Ind_{L_0(c,b_0)}\}_{b_0\in \P R_0^{n-1}})$ is just the dimension of the subspace spanned by $\{\Ind_{L_0(c,b_0)}\}_{b_0\in \P R_0^{n-1}}$. Using Lemma \ref{lem:rankSizeLine} and the induction hypothesis we have,

\begin{align}
    \crank(\{\Ind_{L_0(c,b_0)}\}_{b_0\in \P R_0^{n-1}})\ge \frac{\left|\bigcup_{b_0\in \P R_0^{n-1}} L_0(c,b_0) \right|}{|R_0|}\ge \frac{N_0^{n-1}}{\prod\limits_{i=2}^{r+1} \left(2-p_i^{-1}\right)^{n}}.\label{eq:crankL0}
\end{align}

\end{proof}

\begin{claim}[crank bound for $\{D_{b_1}\}_{b_1\in \P\F_{p_1}^{n-1}}$]\label{claim:crankBparameterBound}
\em{$$\crank(\{D_{b_1}\}_{b_1\in \P\F_{p_1}^{n-1}})\ge \delta_{n,kp_1-1}.$$}
\end{claim}
\begin{proof}
To lower bound $\crank(\{D_{b_1}\}_{b_1\in \P\F_{p_1}^{n-1}})$ we will examine the matrix $D_{\P\F_{p_1}^{n-1}}$ obtained by concatenating all the matrices in $\crank(\{D_{b_1}\}_{b_1\in \P\F_{p_1}^{n-1}})$ along their columns. We see that this is precisely the matrix for the map $\text{EVAL}_{\P \F_{p_1}^{n-1}}^k$ restricted on the input space $\F_{p_1}[x_1,\hdots,x_n]_{=kp_1-1}$.

We claim $D_{\P\F_{p_1}^{n-1}}$ is injective over $\F_{p_1}[x_1,\hdots,x_n]_{=kp_1-1}$. If a polynomial $f\in\F_{p_1}[x_1,\hdots,x_n]_{=kp_1-1}$ lies in the kernel of $D_{\P\F_{p_1}^{n-1}}$ then that means that $f$ and all its Hasse derivatives of weight at most $k$ vanish over $\P\F_{p_1}^{n-1}$. As $f$ is homogenous, all its Hasse derivatives will be too. This means that $f$ and all its Hasse derivatives of weight at most $k$ vanish over all of $\F^n_{p_1}$. This means $f$ vanishes with multiplicity at least $k$ over all of $\F^n_{p_1}$. By Lemma \ref{multSchwartz} we see that $f$ must be identically zero. Hence, we have the desired injectivity.

As $\F_{p_1}[x_1,\hdots,x_n]_{=kp_1-1}$ has dimension $\delta_{n,kp_1-1}$ we see that $D_{\P\F_{p_1}^{n-1}}$ must be of rank at least $\delta_{n,kp_1-1}$. This also gives us the desired crank bound.
\end{proof}

Using Lemma \ref{lem:crankTensorProduct} and equation \eqref{eq:Dbrankeq} with the bounds from Claim \ref{claim:crankCparameterBound} and Claim \ref{claim:crankBparameterBound} we have,
\begin{align*} 
\crank\{C^k_{L_1(b)}\otimes \Ind_{L_0(b)}\}_{b\in \P R^{n-1}}&\ge  \crank\{D_{b_1}\otimes \Ind_{L_0(b_1,b_0)}\}_{(b_1,b_0)\in \P \F_{p_1}^{n-1}\times \P R_0^{n-1}}\\ 
&\ge \frac{N_0^{n-1}}{\prod\limits_{i=2}^{r+1} \left(2-p_i^{-1}\right)^{n}}\cdot \delta_{n,kp_1-1}.
\end{align*}

Using Lemma \ref{lem:rankSizeProject} we have,

\begin{align}
    |S|\binom{2k-k/p_1+n-1}{n}\ge \frac{N_0^{n-1}}{\prod\limits_{i=2}^{r+1} (2-1/p_i)^{n}}\binom{kp_1+n-2}{n-1}.\label{eq:finalSbound}
\end{align}

To get the right bound we assume $k$ is a perfect square and apply \eqref{eq:finalSbound} on the set $S^{\sqrt{k}}$ which is the product of $S$ with itself $\sqrt{k}$ times. It is going to be a Kakeya set in $R^{n\sqrt{k}}$ by Lemma~\ref{lem:productkakeya}. Applying the bound for $S^{\sqrt{k}}$ we have,
$$|S|^{\sqrt{k}}\binom{2k-k/p_1+\sqrt{k}n-1}{\sqrt{k}n}\ge \frac{N_0^{\sqrt{k}n-1}}{\prod\limits_{i=2}^{r+1} (2-1/p_i)^{\sqrt{k}n}}\binom{kp_1+\sqrt{k}n-2}{\sqrt{k}n-1}.$$
Rearranging the terms we have
$$|S|^{\sqrt{k}}\ge \frac{N'^{\sqrt{k}n-1}}{\prod\limits_{i=2}^{r+1} (2-1/p_i)^{\sqrt{k}n}}\frac{(kp_1+\sqrt{k}n-2)\hdots (kp_1)}{(2k-k/p_1+\sqrt{k}n-1)\hdots (2k-k/p_1+1)}\frac{\sqrt{k}n}{2k-k/p_1+\sqrt{k}n-1}.$$
Taking $\sqrt{k}$th root on both sides and letting $k$ grow to infinity in the set of numbers which are square multiples of $p_1$ gives the desired result.

\end{proof}

\section{Kakeya Sets over $\Z/p^k\Z$}\label{sec:powConj}
In this section we prove Theorem~\ref{thm:reduceRankpk}. Recall that $W_{p^k,n}$ is the $p^{kn} \times p^{kn}$ point-hyperplane incidence matrix defined in the introduction.  We want to show the size of any Kakeya set $ S \subset (\Z/p^k\Z)^n$ is lower bounded by the $\F_p$ rank of  $W_{p^k,n}$. 

We start with a lemma that generalizes the following simple observation. Suppose we have a matrix $M$ with entries in the set $\{0,1,-1\}$ and let $\hat M$ be the same matrix with all $-1$ entries replaced with $1$s. It is easy to see that $$\rank_{\C}(M) \geq \rank_{\F_2}(\hat M)$$ since, if some sub-determinant in $\hat M$ is not zero over $\F_2$ then the corresponding determinant in $M$ cannot be zero over the complex numbers. This trivial claim has the following less trivial generalization involving roots of unity of any prime power order.

\begin{lem}\label{lem:rankCtoF}
Let $\gamma$ be a complex primitive $p^k$th root of unity for prime $p$ and natural number $k$. Let $M$ be a matrix whose entries belong to the set $\{0,1,\gamma,\hdots,\gamma^{p^k-1}\}$. Let $\hat M$ be a matrix of the same dimensions of $M$ and with entries
$$ \hat M_{ij} = \left\{
	\begin{array}{ll}
		0  & \mbox{if } M_{i,j}=0 \\
		1 & \mbox{otherwise} 
	\end{array}
\right. $$
Then we have,
\em{$$\text{rank}_{\C}(M)\ge \text{rank}_{\F_p}(\hat{M}).$$}.
\end{lem}
\begin{proof}

Let $Q(x)$ be a matrix obtained from $M$ by replacing $\gamma$ with a formal variable $x$ so that  $Q(\gamma)=M$ and $Q(1)=\hat{M}$. Let $f(x)$ be the determinant of some  sub matrix of $Q(x)$. If the corresponding sub-determinant of $M$ is zero then we have $f(\gamma)=0$. As $f$ has integer coefficients it  must be divisible by the minimal polynomial of $\gamma$ which is $$ m(x)= \frac{x^{p^k}-1}{x^{p^{k-1}}-1} = 1  +x^{p^{k-1}} + x^{2p^{k-1}} + \ldots + x^{(p-1)p^{k-1}}.$$ Hence $f(1)=0$ modulo $p$ and so the coresponding sub-determinant is zero also in $\hat M$ when computed over $\F_p$. This proves the lemma since a non-zero  $r \times r$ determinant in $\hat M$ over $\F_p$ implies the same non-zero determinant in $M$ (over $\C$).
\end{proof}

We are now ready to prove Theroem~\ref{thm:reduceRankpk}, stated here again for convenience.

\begin{repthm}{thm:reduceRankpk}
Given a prime $p$ and numbers $k,n$, every Kakeya set $S$ in $(\Z/p^k\Z)^n$ satisfies,
\em{$$|S|\ge \text{rank}_{\F_p}(W_{p^k,n}).$$}
\end{repthm}
\begin{proof}
Let $R=\Z/p^k\Z$. Given a Kakeya set $S$ in $R^n$ consider the line-matrix  $M_S$ given by Definition~\ref{def:linematrix}. Let $\gamma$ be a primitive $p^k$th root of unity in $\C$ and let $F$ be the $p^{kn} \times p^{kn}$ matrix whose rows and columns are indexed by $R^n$ and whose entry in position $(i,j) \in R^n \times R^n$ is $$ F_{i,j} = \gamma^{\langle i, j \rangle}$$ ($F$ is the  complex Discrete Fourier Transform matrix for the group $R^n$). 

We consider the product $M_S\cdot F$ over $\C$. Given a direction $b\in \P R^{n-1}$ and a point $x(b)$ consider a line $$L(b)=\{x(b)+tb|t\in R\}$$ in that direction contained in $S$. When we multiply the row $\Ind_{L(b)}$ of $M_S$ with $F$, the $j$'th coordinate,  for $j \in R^n$, is given by
$$(\Ind_{L(b)}\cdot F)_j=\sum\limits_{t\in R}\gamma^{\langle x(b)+tb,j\rangle}=\left\{
	\begin{array}{ll}
		0  & \mbox{if } \langle b,j\rangle\ne 0 \mod p^k \\
		p^k\gamma^{\langle x(b),j\rangle} & \mbox{otherwise} 
	\end{array}
\right. .$$
A lower bound on the rank of $M_S\cdot F$ over $\C$ will give a lower bound on the rank of $M_S$ which, using Lemma \ref{lem:rankSizeLine}, will give us a lower bound on $|S|$.
Applying Lemma \ref{lem:rankCtoF} we have that the rank of $M=p^{-k}(M_S\cdot F)$ over $\C$ is lower bounded by the rank of $\hat{M}$ over $\F_p$ where the entry in position $(b,j)\in R^n\times R^n$  of $\hat{M}$ is
$$\hat{M}_{(b,j)} = \left\{
	\begin{array}{ll}
		0  & \mbox{if } \langle b,j\rangle\ne 0 \mod p^k \\
		1 & \mbox{otherwise} 
	\end{array}
\right. .$$
This completes the proof since the matrix  $\hat{M}$ has the same rows as the matrix $W_{p^k,n}$ (recall that we defined $W_{p^k,n}$ to have repeated rows corresponding to different scaling of the same vector). This completes the proof.
\end{proof}

This reduces the question of lower bounding the size of Kakeya sets in $(\Z/p^k\Z)^n$ to the estimation of the rank of a concrete matrix, independent of the set $S$. In the next section we show that the matrix $W_{p^k,n}$ contains an identity matrix of size roughly $p^{kn/2}$. 

\subsection{Rank of $W_{p^k,n}$ via Matching Vector families}

\begin{define}[Matching Vectors]
A {\em matching vector} (MV) family over $(\Z/N\Z)^n$ is a pair $(U,V)$ such that 
$U=(u_1,\hdots,u_m)$ and $V=(v_1,\hdots,v_m)$ with each $u_i$ and each $v_j$ belonging to the set $(\Z/N\Z)^n$ and such that 
$\langle u_i,v_j\rangle  = 0 \mod N$ iff $i=j$.
\end{define}

The following simple observation relates the size of an MV family with the rank of the point-hyperplane incidence matrix.

\begin{lem}\label{lem:matchingReduction}
If there exists a matching vector family of size $m$ over $(\Z/p^k\Z)^n$ then, for any field $\F$, we have
\em{$$\text{rank}_{\F}(W_{p^k,n})\ge m.$$}
\end{lem}
\begin{proof}
The sub-matrix corresponding to the rows labeled by $U$ and the columns labelled by $V$  in $W_{p^k,n}$ is the identity matrix (after reordering).
\end{proof}

Constructions of large MV families have found surprising application in combinatorics and theoretical computer science. In particular, constructions of MV families for small composite $N$ (even for $N=6$) and growing $n$ given by Grolmusz \cite{grolmusz2000superpolynomial} have found a surprisingly large number of applications. For our purposes, when $N$ is a large prime power, we will use a less known construction originally given in  \cite{dvir2011matching} and improved by \cite{YUAN2012494}.

\begin{thm}[\cite{YUAN2012494}]
For every integer $n$ and any sufficiently large $N$, there exist MV families over  $(\Z/N\Z)^n$ of size at least,
$$\left(\frac{N}{n-2}\right)^{(n-2)/2}.$$
\end{thm}

Combining this theorem (with $N = p^k$) with Lemma \ref{lem:matchingReduction} and Theorem \ref{thm:reduceRankpk} gives us the  a lower  bound for Kakeya sets in $(\Z/p^k\Z)^n$ of the order of  $p^{kn/2}$ which is worse than the best known bounds. Using larger MV families to lower bound the rank of $W_{p^k,n}$ cannot lead to significantly stronger bounds as it is shown in \cite{dvir2013matching} that those cannot be larger than $N^{n/2 + O(1)}$ for any $N$.
\bibliographystyle{alpha}
\bibliography{bibliography}

\begin{thebibliography}{DKSS13}

\bibitem[Bou99]{Bou99}
J.~Bourgain.
\newblock On the dimension of {K}akeya sets and related maximal inequalities.
\newblock {\em Geom. Funct. Anal.}, 9(2):256--282, 1999.

\bibitem[Car18]{CML_2018__10_1_3_0}
Xavier Caruso.
\newblock Almost all non-archimedean kakeya sets have measure zero.
\newblock {\em Confluentes Mathematici}, 10(1):3--40, 2018.

\bibitem[DGY11]{dvir2011matching}
Zeev Dvir, Parikshit Gopalan, and Sergey Yekhanin.
\newblock {M}atching {V}ector codes.
\newblock {\em SIAM Journal on Computing}, 40(4):1154--1178, 2011.

\bibitem[DH13a]{dummithablicsek2013}
Evan~P. Dummit and Márton Hablicsek.
\newblock Kakeya sets over non-archimedean local rings.
\newblock {\em Mathematika}, 59(2):257–266, 2013.

\bibitem[DH13b]{dvir2013matching}
Zeev Dvir and Guangda Hu.
\newblock {M}atching-{V}ector families and {LDC}s over large modulo.
\newblock In {\em Approximation, Randomization, and Combinatorial Optimization.
  Algorithms and Techniques}, pages 513--526. Springer, 2013.

\bibitem[DKSS13]{dvir2013extensions}
Zeev Dvir, Swastik Kopparty, Shubhangi Saraf, and Madhu Sudan.
\newblock Extensions to the method of multiplicities, with applications to
  {K}akeya sets and mergers.
\newblock {\em SIAM Journal on Computing}, 42(6):2305--2328, 2013.

\bibitem[Dvi09]{dvir2009size}
Zeev Dvir.
\newblock On the size of kakeya sets in finite fields.
\newblock {\em Journal of the American Mathematical Society}, 22(4):1093--1097,
  2009.

\bibitem[Dvi10]{dvir2010incidence}
Zeev Dvir.
\newblock Incidence theorems and their applications.
\newblock {\em Foundations and Trends in Theoretical Computer Science},
  6(4):257--293, 2010.

\bibitem[EOT10]{ellenberg2010kakeya}
Jordan~S Ellenberg, Richard Oberlin, and Terence Tao.
\newblock The {K}akeya set and maximal conjectures for algebraic varieties over
  finite fields.
\newblock {\em Mathematika}, 56(1):1--25, 2010.

\bibitem[Fra16]{Fraser_2016}
Robert Fraser.
\newblock {KAKEYA}-{TYPE} {SETS} {IN} {LOCAL} {FIELDS} {WITH} {FINITE}
  {RESIDUE} {FIELD}.
\newblock {\em Mathematika}, 62(2):614--629, jan 2016.

\bibitem[GD68]{goethalsDel1968}
J.~. {Goethals} and P.~{Delsarte}.
\newblock On a class of majority-logic decodable cyclic codes.
\newblock {\em IEEE Transactions on Information Theory}, 14(2):182--188, 1968.

\bibitem[Gro00]{grolmusz2000superpolynomial}
Vince Grolmusz.
\newblock Superpolynomial size set-systems with restricted intersections mod 6
  and explicit {R}amsey graphs.
\newblock {\em Combinatorica}, 20(1):71--86, 2000.

\bibitem[HR17]{zbMATH02610444}
G.~H. {Hardy} and S.~{Ramanujan}.
\newblock {The normal number of prime factors of a number \(n\)}.
\newblock {\em {Quart. J.}}, 48:76--92, 1917.

\bibitem[HW18]{hickman2018fourier}
Jonathan Hickman and James Wright.
\newblock The {F}ourier restriction and {K}akeya problems over rings of
  integers modulo {N}.
\newblock {\em Discrete Analysis}, (11), 2018.

\bibitem[KT02]{katz2002new}
Nets~Hawk Katz and Terence Tao.
\newblock New bounds for {K}akeya problems.
\newblock {\em Journal d'Analyse Math{\'e}matique}, 87(1):231--263, 2002.

\bibitem[LV14]{LandjevV14}
Ivan~N. Landjev and Peter Vandendriessche.
\newblock On the rank of incidence matrices in projective {H}jelmslev spaces.
\newblock {\em Des. Codes Cryptogr.}, 73(2):615--623, 2014.

\bibitem[MM68]{macwilliams1968p}
Florence~Jessie MacWilliams and Henry~B Mann.
\newblock On the p-rank of the design matrix of a difference set.
\newblock {\em Information and Control}, 12(5):474--488, 1968.

\bibitem[Sch79]{schwartz1979probabilistic}
Jacob~T Schwartz.
\newblock Probabilistic algorithms for verification of polynomial identities.
\newblock In {\em International Symposium on Symbolic and Algebraic
  Manipulation}, pages 200--215. Springer, 1979.

\bibitem[Smi69]{smith1969p}
KJC Smith.
\newblock On the p-rank of the incidence matrix of points and hyperplanes in a
  finite projective geometry.
\newblock {\em Journal of Combinatorial Theory}, 7(2):122--129, 1969.

\bibitem[SS08]{saraf2008}
Shubhangi Saraf and Madhu Sudan.
\newblock An improved lower bound on the size of {K}akeya sets over finite
  fields.
\newblock {\em Anal. PDE}, 1(3):375--379, 2008.

\bibitem[Wol99]{wolf1999}
Thomas Wolff.
\newblock Recent work connected with the kakeya problem.
\newblock {\em Prospects in mathematics (Princeton,NJ, 1996)}, pages 29--162,
  1999.

\bibitem[YGK12]{YUAN2012494}
Chen Yuan, Qian Guo, and Haibin Kan.
\newblock A novel elementary construction of matching vectors.
\newblock {\em Information Processing Letters}, 112(12):494 -- 496, 2012.

\bibitem[Zip79]{ZippelPaper}
Richard Zippel.
\newblock Probabilistic algorithms for sparse polynomials.
\newblock In Edward~W. Ng, editor, {\em Symbolic and Algebraic Computation},
  pages 216--226, Berlin, Heidelberg, 1979. Springer Berlin Heidelberg.

\end{thebibliography}

\end{document}